\documentclass{amsart}
\usepackage{amsfonts}
\usepackage{amsmath}
\usepackage{amssymb}
\usepackage{amsthm}
\usepackage{fullpage}
\usepackage{graphicx}
\usepackage{amsfonts}
\usepackage{xcolor}
\usepackage{hyperref}
\usepackage{enumerate}

\numberwithin{equation}{section}

\def\F{{\mathcal F}}
\def\bd{{\boldsymbol{\cdot}}}

\def\N{\mathbb N}

\newtheorem{theorem}{Theorem}[section]
\newtheorem{proposition}[theorem]{Proposition}
\newtheorem{lemma}[theorem]{Lemma}
\newtheorem{corollary}[theorem]{Corollary}
\newtheorem{conjecture}[theorem]{Conjecture}

\newtheorem{remark}[theorem]{Remark}

\theoremstyle{definition}

\begin{document}

%
%

\title[Weighted consecutive Davenport constant]{A note on weighted consecutive Davenport constant}

\author{A.  Lemos$^{1}$, A.O. Moura$^{1}$, S. Ribas$^{2}$  and A.T. Silva$^{1}$}
\thanks{All the authors were partially supported by FAPEMIG grants APQ-02546-21 and RED-00133-21. Ribas was also supported by FAPEMIG grant APQ-01712-23}
\address{$^1$Departamento de Matem\'{a}tica, Universidade Federal de Vi\c cosa, Vi\c cosa-MG, Brazil}
\email{\url{abiliolemos@ufv.br} \\ \url{allan.moura@ufv.br} \\ \url{anderson.tiago@ufv.br}}
\address{$^2$Departamento de Matem\'{a}tica, Universidade Federal de Ouro Preto, Ouro Preto-MG, Brazil}
\email{\url{savio.ribas@ufop.edu.br}}



\begin{abstract}
Let $G$ be a 
group  and $A\subseteq [1,\exp(G)-1]$. We define the constant ${\sf C}_A(G),$ which is the least positive integer $\ell$ such that every sequence over $G$ of length at least $\ell$ has an $A$-weighted consecutive product-one subsequence. In this paper, among other things, we prove that ${\sf C}_A(C_n^2)=4$ with $A=[1,n-1],$ and ${\sf C}(H\times K)=|H||K|$, where $H$ is a finite abelian group and $K$ is a metacyclic group.
\end{abstract}

\maketitle

\section{Introduction}

Let $G$ be a finite abelian group with exponent $n$ and $S$ be a sequence over $G$. The enumeration of subsequences with certain prescribed properties is a classical topic in Combinatorial Number Theory, going back to Erd\H{o}s, Ginzburg and Ziv (see \cite{EGZ,Ger1,Ger2}) who proved that $2n-1$ is the smallest positive integer such that every sequence $S$ over a cyclic group $C_{n}$ of order $n$, with length $2n-1,$ has a zero-sum subsequence of length $n$. This raises the problem of determining the smallest positive integer $\ell$ such that every sequence $S=g_{1} \bd \dots \bd g_{\ell}$ has a nonempty zero-sum subsequence. Such an integer $\ell$ is called the {\it Davenport constant of $G$} (see \cite{Dav, OlsonI}), denoted by ${\sf D}(G)$, which is still unknown for a wide class of groups. Similarly, for a nonempty subset $A\subset \mathbb{Z}$, Adhikari \textit{et al.} (see \cite{Adh1}) defined the {\it $A$-weighted Davenport constant}, denoted by ${\sf D}_A(G)$, to be a smallest $t\in \mathbb{N}$ such that every sequence $S$ over $G$ of length $t$ has nonempty $A$-weighted zero-sum subsequence. For further results on weighted Davenport constant, see \cite{BR,AGS,losm,losm2,T}. Often these problems are studied for commutative groups using additive notation, because this one calls such a subsequence a ``zero-sum'' subsequence. The current paper focuses on the direct product of commutative and non-commutative groups noted multiplicatively, thus we use the term ``product-one'' sequence instead. In 2023, Mondal, Paul and Paul (see \cite{mon}) defined the constant ${\sf C}_A(G),$ which is the smallest natural $\ell$ such that every sequence over $G,$ with length at least $\ell$ has an $A$-weighted consecutive product-one subsequence. They achieved several results for the group $C_n$ with varying weights. Here, we extend several of their results to more general groups.

\section{Notation and Preliminary Results}

Let $\mathbb{N}_{0}$ be {\it the set of non-negative integers}. For integers $a,b\in\mathbb{N}_{0}$, we define $[a,b]=\left\{ x\in\mathbb{N}_{0}:a\leq x\leq b\right\} $. Let $G$ be a multiplicative group. Let $\F^*(G)$ denote the free non-abelian monoid with basis $G$, whose elements are also called the {\it ordered sequences over $G$}. The set $\F^*(G)$ can be seem as the semigroup of words over the alphabet $G$, and its elements are also called {\it words} or {\it strings}. If $S = g_1 \bd {\dots} \bd g_k\in \F^*(G)$ (i.e. $S$ is an ordered sequence) and $\ell \in \N$, then $\pi^*(S) = g_1 {\dots} g_k \in G$ is the {\it ordered product of $S$}, and $\left|S\right|=k$ is the {\it length of $S$}. Furthermore, we use the notation $S^{[\ell]} = \underbrace{(g_1 \bd {\dots} \bd g_k) \bd {\dots} \bd (g_1 \bd {\dots} \bd g_k)}_{\ell \text{ copies of $S$ in this order}}$.  A {\it subsequence} $T=g_{i_{1}\bd {\dots} \bd}g_{i_{m}}$ of $S$, with $\{i_{1},\ldots,i_{m}\}\subseteq[1,k]$ is denoted by $T\mid S$. An {\it ordered subsequence} $T$ of $S$ is an ordered sequence of the form $T = g_i \bd g_{i+1} \bd {\dots} \bd g_{j-1} \bd g_j$ for some $1 \le i \le j \le k$. We also write $T \mid S$ to indicate that $T$ is an ordered subsequence of $S$.

Let $\mathcal{F}\left(A\right)$ be the free abelian monoid with basis $A$. In other words, $\mathcal{F}\left(A\right)$ is the set of (unordered) {\it sequences over $A$}. Fixed $\mathbf{a} = a_{1}\bd {\dots} \bd a_{k}\in\mathcal{F}(A)$, the product of the form $\pi^{\mathbf{a}}\left(S\right) = g_1^{a_{1}} \bd {\dots} \bd g_k^{a_{k}}$ is an {\it $A$-weighted consecutive product}. If $n$ is the exponent of $G$ and $A=[1,n-1]$, we say that $A$ is the {\it full weight}.

We define $\Pi_{A}\left(S\right) = \left\{ g_i^{a_{i}}g_{i+1}^{a_{i+1}} \bd {\dots} \bd g_{j-1}^{a_{j-1}} g_j^{a_{j}}: 1 \le i \le j \le k\mbox{ and }a_{t}\in A \mbox{ for all } i \le t \le j \right\} $ to be the {\it set of nonempty $A$-weighted consecutive subproducts of $S$}.
According to the previous definitions, we adopt the convention that $\pi^{\mathbf{a}}\left(\lambda\right)=1$,
for any $\mathbf{a}\in\mathcal{F}(A),$ where $\lambda$ denote the empty sequence. For convenience, we define
$\Pi_{A}^{\bullet}\left(S\right)=\Pi_{A}\left(S\right)\cup\left\{ 1\right\}$.

The sequence $S$ is called 
\begin{enumerate}[(i)]
\item an {\it $A$-weighted consecutive product-one sequence} if $\pi^{\mathbf{a}}\left(S\right)=1$ for some $\mathbf{a}\in\mathcal{F}(A)$, and
\item an {\it $A$-weighted consecutive product-one free sequence} if $1\notin\Pi_{A}\left(S\right).$
\end{enumerate}
When $A=\{1\},$ we call $S$ {\it a consecutive product-one sequence} and {\it a consecutive product-one free sequence}, respectively.

Let ${\sf C}_A(G)$ be the smallest positive integer $\ell$ such that every sequence over $G,$ with length at least $\ell,$ has an $A$-weighted consecutive product-one subsequence. This constant is called {\it consecutive $A$-weighted Davenport constant of $G$}. For $A = \{1\}$, we omit the index $A$. The following lemmas were previously established.

\begin{lemma} \cite[Theorem 1]{mon}
\label{ub}
Let $(G, \cdot)$ be a finite group with $|G| =m,$ $k \ge m,$ and identity $1.$ Then given any sequence $S = g_1 \bd {\dots} \bd g_k \in \F^*(G)$ of length $k,$ there exist $i, j \in [1, k]$ such that $i\le j$ and $g_ig_{i+1}\cdots g_j=1.$ 
\end{lemma}

\begin{lemma}\cite[Corollary 1]{mon}
\label{cyclic}
We have ${\sf C}(C_n)=n.$
\end{lemma}

Next, we will establish our first result, which is a lower bound on ${\sf C}_A(G)$, where $G$ is a direct product of two groups. A similar result for the direct product of two modules was shown in \cite{mpp}.

\begin{proposition}\label{induction}
Let $(H, \cdot), (K, \cdot)$ be finite groups with $\exp(H) = \exp(K)$ and $A \subseteq [1,\exp(K)-1].$ Then ${\sf C}_A(H \times K)\ge {\sf C}_A(H){\sf C}_A(K).$ In particular, when $A=\{1\}$ the hypothesis $\exp(H)=\exp(K)$ can be removed.
\end{proposition}
\begin{proof}
Set $t_1={\sf C}_A(H)-1$ and $t_2= {\sf C}_A(K)-1$. By definition, there exist $A$-weighted consecutive product-one free sequences $S_1' = h_1 \bd {\dots} \bd h_{t_1} \in \F^*(H)$ over $H$ and $S_2' = g_1 \bd {\dots} \bd g_{t_2} \in \F^*(K)$ over $K$. Let $S_1 = (h_1,1) \bd {\dots} \bd (h_{t_1},1)$ and $S_2 = (1,g_1) \bd {\dots} \bd (1,g_{t_2})$. Both $S_1$ and $S_2$ belong to $\F^*(H \times K)$ and are $A$-weighted consecutive product-one free sequence, since $\exp(H)=\exp(K)$ (it is clear that if $A=\{1\}$, then this happens without the condition $\exp(H)=\exp(K)$). We construct the following sequence
\begin{align*}
	S = (h_1,1) &\bd {\dots} \bd (h_{t_1},1)\bd (1,g_1)\bd \\
	(h_1,1) &\bd {\dots} \bd (h_{t_1},1)\bd (1,g_2)\bd \\
	&\quad \!\vdots \\
	(h_1,1) &\bd {\dots} \bd (h_{t_1},1)\bd (1,g_{t_2})\bd \\
	(h_1,1) &\bd {\dots} \bd (h_{t_1},1).
\end{align*}
Notice that $|S|=t_2(t_1+1)+t_1={\sf C}_A(H){\sf C}_A(K)-1.$ As $S_1$  is an $A$-weighted consecutive product-one free sequence, we observe that if there exists $T\mid S$ such that $T$ is an $A$-weighted consecutive product-one sequence, then $(1,g_{j_1}) \bd {\dots} \bd (1,g_{j_l})\mid T$ with $1\le j_i\le t_2,$ $j_{i+1}=j_{i}+1,$ and $(1,g_{j_1}) \bd {\dots} \bd (1,g_{j_l})$ is an $A$-weighted consecutive product-one subsequence, but this is a contradiction since $S_2$ is an $A$-weighted consecutive product-one free sequence. It follows that $S$ is an $A$-weighted consecutive product-one free sequence over $H \times K$. Therefore, ${\sf C}_A(H \times K)\ge {\sf C}_A(H){\sf C}_A(K).$
\end{proof}

\begin{remark}
We note that in the proof of the Proposition \ref{induction}, $t_1={\sf C}_A(H)-1=0$ if and only if $\exp(H)\in A.$ Thus, ${\sf C}_A(H \times K)\ge {\sf C}_A(K)$ provided $\exp(H)<\exp(K).$
\end{remark}

\section{Results when $A=\{1\}$}

In this section, we present two results for the case $A = \{1\}$, that is, for the unweighted case. The first one determines the value of ${\sf C}(G)$ for every finite abelian group $G$.

\begin{theorem}\label{ab}
Let $G$ be finite abelian group. Then ${\sf C}(G) = |G|$.
\end{theorem}

\proof
We write $G=C_{n_1}\times C_{n_2}\times \cdots\times C_{n_r},$ where $1<n_1 \mid n_2\mid \cdots\mid n_r=\exp(G).$ By Lemma \ref{ub}, we have ${\sf C}(G) \le|G|.$ To prove that ${\sf C}(G)\ge |G|$ we proceed by induction on $r.$ The case $r=1$ follows from Lemma \ref{cyclic}. We write $G=H\times C_{n_r}.$ By induction hypothesis, it follows that ${\sf C}(H)\ge |H|.$ According to Proposition \ref{induction}, we have ${\sf C}(G)\ge |G|.$ Therefore, we are done. 
\qed

\vspace{2mm}

The second result of this section concerns metacyclic groups. It is known (see \cite{Hem}) that if $G$ is a metacyclic group, then
\begin{equation}\label{eq:Gmetacyclic}
G = \langle x, y \mid y^n = 1, x^k = y^\ell, yx = xy^s \rangle,
\end{equation}
where $n \mid s^k - 1$ and $n \mid \ell(s-1)$. After a change of variables, one may assume without loss of generality that $\ell \mid n$. Moreover, $G = \{x^a y^b; a \in [0,k-1], b \in [0,n-1]\}$, therefore $|G| = nk$.

\begin{theorem} \label{mc}
Let $G$ be a metacyclic group. Then ${\sf C}(G) = |G|$.
\end{theorem}

\proof
Since ${\sf C}(G) \le |G|$ for every finite group $G$, it is enough to prove that ${\sf C}(G) \ge |G|$. We write $G$ as in Eq. \eqref{eq:Gmetacyclic}. Let $S \in \F^*(G)$ given by $S = (y^{[n-1]} \bd x)^{[k-1]} \bd y^{[n-1]}$, so that $|S| = nk-1$. We claim that $S$ is a consecutive product-one free sequence. In fact, if $T \mid S$ is a consecutive product-one subsequence, then we must have $x \mid T$, otherwise $T \mid y^{[n-1]}$, a contradiction. Since $x \mid T$, it follows that $\pi^*(T) = x^u y^v$ for some $u \in [1,k-1]$, therefore $\pi^*(T) \neq 1$, another contradiction.
\qed

\begin{corollary}
Let $G=H\times K$ be a finite group, where $H$ is a finite abelian group and $K$ is a metacyclic group. Then, ${\sf C}(G)=|H||K|.$
\end{corollary}

\proof
This is an immediate consequence of Lemma \ref{ub}, Proposition \ref{induction}, and Theorems \ref{ab} and \ref{mc}.
\qed

\vspace{2mm}

The results presented in this section suggest the following problem.

\begin{conjecture}
For every finite group $G$, it holds ${\sf C}(G) = |G|$.
\end{conjecture}

It is worth mentioning that for infinite groups, either there exists an infinite-order element or an infinite set of independent elements. Anyway, it is easy verifying that ${\sf C}(G) = \infty$ for every infinite group $G$.

\section{Results when $A\subseteq[1,\exp(G)-1]$} 

Let $U(n)$ denote the multiplicative group of units in the ring $\mathbb{Z}_n.$ For $\nu\ge1,$ let $U(n)^{\nu} = \{x^{\nu}; x \in U(n)\}.$ In this section, we extend some results of \cite{mon} to larger classes of finite abelian groups, considering certain weights. First, we will need the following. 

\begin{lemma} \cite[Theorems 2, 4 and 7]{mon}
\label{completo}
\begin{enumerate}[(a)]
\item If $A=[1,n-1],$ then ${\sf C}_A(C_n)=2.$
\item Let $p$ an odd prime number and $A=U(p)^2.$ Then ${\sf C}_A(C_p)=3.$
\item Let $p$ an odd prime number and $A=U(p)\backslash U(p)^2.$ Then ${\sf C}_A(C_p)=3.$
\item Let $p$ a prime number such that $p\equiv1 \pmod3$, $p\neq7,$ and $A=U(p)^3.$ Then ${\sf C}_A(C_p)=3.$
\end{enumerate}
\end{lemma}

Next we consider the same weights as in the previous lemma and we present lower bounds for larger rank groups. The first one will be useful to find the value of ${\sf C}_A(C_n^2)$, where $A$ is the full weight.

\begin{proposition}\label{lb}\hspace{2cm}
\begin{enumerate}[(a)]
\item If $A=[1,n-1],$ then ${\sf C}_A(C_n^r)\ge2^r.$
\item Let $p$ an odd prime number and $A=U(p)^2.$ Then ${\sf C}_A(C_p^r)\ge3^r.$
\item Let $p$ an odd prime number and $A=U(p)\backslash U(p)^2.$ Then ${\sf C}_A(C_p^r)\ge3^r.$
\item Let $p$ a prime number such that $p\equiv1 \pmod3$, $p\neq7,$ and $A=U(p)^3.$ Then ${\sf C}_A(C_p^r)\ge3^r.$
\end{enumerate}\end{proposition}
\proof
In all items we proceed by induction on $r$, and all the basis are given by the previous lemma. For $A=[1,n-1]$, we write $C_n^r = H \times C_n.$ By induction hypothesis we have ${\sf C}_A(H) \ge 2^{r-1}.$ Therefore, according to Proposition \ref{induction} we have ${\sf C}_A(C_n^r)\ge 2^r,$ which proves {\it (a)}. Similarly, for $A=U(p)^2$, we write $C_p^r = H\times C_p.$ By induction hypothesis we have ${\sf C}_A(H)\ge 3^{r-1}.$ Therefore, according to Proposition \ref{induction} we have ${\sf C}_A(C_p^r)\ge 3^r$, which proves {\it (b)}. Items {\it (c)} and {\it (d)} are completely similar.
\qed

\vspace{2mm}

The rest of this section is dedicated to find the value of ${\sf C}_A(C_n^2)$, where either $A$ is the full weight or $A=\{1,2,\ldots,d^kn-1\}\setminus\{d^{k-i}n:i\in[1,k]\}$. First we need the following result.

\begin{lemma}\cite[Theorem 5.2]{marc}
 \label{davenport}
 Let $G=H\times C_{n}^{r}$, where $H=C_{n_{1}}\times\cdots\times C_{n_{t}}$ with
$1<n_{1} \mid n_{2} \mid \cdots \mid n_{t} \mid n = \exp(G),$ $n_{t}<n$ and $A=\{1, \ldots, n-1\}.$ Then, ${\sf D}_{A}(G)=r+1$. 
\end{lemma}

Finally we obtain the following result for the full weight.

\begin{theorem}\label{c2}
If $A=[1,n-1],$ then ${\sf C}_A(C_n^2)=4.$
\end{theorem}
\proof
According to Proposition \ref{lb}(a), we have ${\sf C}_A(C_n^2)\ge4.$ Let $S=g_1 \bd g_2 \bd g_3 \bd g_4$ over $C_n^2.$ Since ${\sf D}_A(C_n^r)=r+1,$ it follows that $T = g_i \bd g_j \bd g_l$ has at least one $A$-weighted product-one subsequence, where $i,j,l\in\{1,2,3,4\}$ are distinct. If such a subsequence is formed by consecutive elements, then we are done. Otherwise, we may suppose that this subsequence is not formed by consecutive elements. We claim that $S$ is an $A$-weighted product-one subsequence. In fact, $T_1 = g_1 \bd g_2 \bd g_3$ and $T_2 = g_2 \bd g_3 \bd g_4$ have each at least one $A$-weighted produt-one subsequence without consecutive elements. The only possibilities are $g_1 \bd g_3$ for $T_1$, say $g_1^{a_1}g_3^{a_3} = 1$ for some $a_1, a_3 \in A$, and $g_2 \bd g_4$ for $T_2$, say $g_2^{a_2}g_4^{a_4} = 1$ for some $a_2, a_4 \in A$. Hence $S$ is an $A$-weighted product-one sequence, since $g_1^{a_1}g_2^{a_2}g_3^{a_3}g_4^{a_4}=1$.
\qed

For $A=\{1,2,\ldots,d^kn-1\}\setminus\{d^{k-i}n:i\in[1,k]\}$, we need the following auxiliar result.

\begin{lemma}\cite[Theorem 3.2]{losm2}\label{d2}
 Let $G=H\oplus C_{d^kn}^r$, where $\exp(H)\mid d^k,$ $\gcd(d,n)\leq d-1$ and $d^kn\geq 6,$ where $k$ is a positive integer. Then ${\sf D}_A(G)=r+1,$ for $A=\{1,2,\ldots,d^kn-1\}\setminus\{d^{k-i}n:i\in[1,k]\}$.
\end{lemma}

Lastly, we obtain the following.

\begin{theorem}
Let $G_r=H\oplus C_{d^kn}^r,$ with $r\in\{1,2\}$, where $\exp(H)\mid d^k,$ $\gcd(d,n)\leq d-1$ and $d^kn\geq 6,$ where $k$ is a positive integer. For $A=\{1,2,\ldots,d^kn-1\}\setminus\{d^{k-i}n:i\in[1,k]\},$ we have
\begin{enumerate}
\item[(i)] ${\sf C}_A(G_1)=2;$
\item[(ii)] ${\sf C}_A(G_2)=4.$
\end{enumerate}
\end{theorem}
\proof
For item $(i),$ clearly ${\sf C}_A(G_1) \ge 2.$ Consider $S=g_1 \bd g_2\in \mathcal{F}(G_1).$ If either $g_1$ or $g_2$ is an $A$-weighted product-one subsequence, then we are done. If does not exist $a_1, a_2\in A$ such that $g_1^{a_1}=1$ and $g_2^{a_2}=1,$ then $g_1 = y^{b_1}$ and $g_2 = y^{b_2}$ for some $b_1, b_2 \in A$ and $y \in G_1.$ Since $-b_1 \in A,$ it follows that $g_1^{b_2}g_2^{-b_1}=1$ in $G_1.$ The proof of item $(ii)$ is similar to proof of Theorem \ref{c2}.

\qed

\end{document}